\documentclass[a4paper,11pt]{article}
\usepackage{graphicx}
\usepackage{epstopdf}
\usepackage{amsthm}
\usepackage[utf8]{inputenc}
\usepackage{mathtools}
\usepackage[english]{babel}
\newtheorem{theorem}{Theorem}[section]
\newtheorem{corollary}{Corollary}[theorem]

\theoremstyle{definition}

\usepackage{amssymb, amsmath}
\title{New class of k-uniformly harmonic functions defined by Al-Oboudi  operator}
\date{}
\begin{document}
\numberwithin{equation}{section}
\maketitle
\date{}
\begin{center}
\author{{\bf{G. M. Birajdar}}\vspace{.31cm}\\
	School of Mathematics \& Statistics,\\
	Dr. Vishwanath Karad MIT World Peace University,\\
	Pune (M.S) India 411038\\
	Email: gmbirajdar28@gmail.com}\vspace{0.31cm}\\
\author{{\bf{N. D. Sangle}}\vspace{0.31cm}\\
Department of Mathematics,\\
D. Y. Patil College of Engineering \& Technology, \\
Kasaba Bawada,
Kolhapur, (M.S.), India 416006\\
Email: navneet\_sangle@rediffmail.com}\vspace{.5cm}\\

\end{center}
\vspace{1cm}
\abstract{}
In this paper, we introduce the class $k$-USH $(u,v,\alpha,\lambda)$ using Al-Oboudi operator which is a subclass of $k$-uniformly harmonic functions. A subclass $k$-UTH $(u,v,\alpha,\lambda)$ of $k$-USH $(u,v,\alpha,\lambda)$ is also been defined and studied in this paper. Extreme points, distortion bounds, convolution condition and convex combination  of functions belonging to the class $k$-UTH $(u,v,\alpha, \lambda)$ are also studied.\\



\section{Introduction} 
\renewenvironment{proof}{{\bfseries Proof:}}{}
Let SH denotes the class of functions $f=h+\overline{g}$ that are harmonic univalent and sense-preserving in the unit disk $U = \left\{ {z \in C:\left| z \right| < 1} \right\}$ for which $f(0)=0, f_z (0)=1$. \\
In \cite{B9} Clunie and Sheil-Small, investigated the class SH as well as its geometric subclasses and its properties. Since then, there have been several studies related to the class SH and its subclasses. Following Jahangiri [\cite{B2}, \cite{B3} ], Silverman \cite{10}, Silverman and Silvia \cite{11}, Öztürk et al. \cite{12} have investigated various subclasses of SH and its properties.\\

For $f=h+\overline{g} \in$ SH the analytic functions $h$ and $g$ may be expressed as 
\begin{align}
h(z) = z + \sum\limits_{n = 2}^\infty  {{a_n}} {z^n}, \quad g(z) = \sum\limits_{n = 1}^\infty  {{b_n}} {z^n}\,\,,\left| {{b_1}} \right| < 1.
\end{align}
Al-Oboudi \cite{B1} introduced the following operator\\
For analytic function $h(z) \in S$, we have
\begin{align*}
{D^0}\,h(z) = h(z).
\end{align*}
\begin{align*}
{D^1}\,h(z) = (1 - \lambda )h(z) + \lambda z\,{h^{'}}(z) = {D_\lambda }h(z),\,\,\,\,\lambda  \ge 0.
\end{align*}
\begin{align*}
{D^u}\,h(z) = {D_\lambda }({D^{u - 1}}\,h(z)),\,\,\,\,\,\,\,(u \in N = \left\{ {1,2,3,...} \right\}).
\end{align*}
and
\begin{align}
{D^u}\,h(z) = z + \sum\limits_{n = 2}^\infty  {{{\left[ {1 + (n - 1)\lambda } \right]}^u}{a_n}} {z^n},\quad \quad u \in {N_0} = N \cup \left\{ 0 \right\}.
\end{align}
Also for harmonic univalent functions $f=h+\overline{g} $, we can have
\begin{align}
{D^u}\,f(z) = {D^u}\,h(z) + \overline {{{( - 1)}^u}{D^u}\,g(z)} \,,\,\,\,\,\,\,\,(u \in {N_0} = N \cup \left\{ 0 \right\}).
\end{align}
where ${D^u}\,h(z) = z + \sum\limits_{n = 2}^\infty  {{{\left[ {1 + (n - 1)\lambda } \right]}^u}{a_n}} {z^n}$ and ${D^u}\,g(z) = \sum\limits_{n = 1}^\infty  {{{\left[ {1 + (n - 1)\lambda } \right]}^u}{b_n}} {z^n}$.\\
For $0 \le \alpha  < 1$, $0 \le k < \infty $, $u>v$, $k$-USH $(u,v,\alpha,\lambda)$ denotes a class of functions $f=h+\overline{g} \in$ satisfying
\begin{align}
{\mathop{\rm Re}\nolimits} \left\{ {\left( {1 + k{e^{i\phi }}} \right)\frac{{{D^u}f(z)}}{{{D^v}f(z)}} - k{e^{i\phi }}} \right\} \ge \alpha. 
\end{align}
Also $k$-UTH $(u,v,\alpha,\lambda)$ subclass of $k$-USH $(u,v,\alpha,\lambda)$ consists of harmonic functions $f_u=h+\overline{g_u}$ so that
\begin{align}
h(z) = z - \sum\limits_{n = 2}^\infty  {\left| {{a_n}} \right|} {z^n}, \quad  {g_u}(z) = {( - 1)^{u - 1}}\sum\limits_{n = 1}^\infty  {\left| {{b_n}} \right|} {z^n},\quad \left| {{b_1}} \right| < 1.
\end{align}
The class $k-USH (u,v,\alpha,\lambda)$ generalizes several classes of harmonic univalent functions defined earlier. For $k=0, u=1, v=0, \lambda=1$, this class reduces to $SH(\alpha)$ the class of univalent harmonic starlike functions of order $\alpha$ which was studied by Jahangiri \cite{B2} and for $k=0, u=2, v=1, \lambda=1$, it reduces to the class $KH(\alpha)$ the class of univalent harmonic convex function of order $\alpha$ which is studied by Jahangiri \cite{B3}. For $k=1, u=1, v=0,\lambda=1$, this class reduces to $G_{H}(\alpha)$ which was studied by Rosy et al. \cite{B7}. For $k=1, u=v+1, \lambda=1$, this class reduces to $RS_{H}(v,\alpha)$ which was studied by Yalcin et al. \cite{B8}. For $ \lambda =1$, this class reduces to $k-USH (u,v,\alpha)$ which was studied by Khan \cite{B5}.

\section{Coefficient Condition and Properties} 

In this section sufficient coefficient inequality for the harmonic function $f$ to be univalent, sense-preserving in the unit disk $U$ and to be in the class $k$-USH $(u,v,\alpha, \lambda)$ are obtained. It is also proved that
this coefficient inequality is necessary for the functions $f$ belonging to the subclass $k$-UTH $(u,v,\alpha, \lambda)$ of $k$-USH $(u,v,\alpha, \lambda)$.

\begin{theorem}
(Sufficient coefficient condition for  $k$-USH $(u,v,\alpha, \lambda)$\\
Let $f=h+\overline{g}$ given by equation (1.1).Furthermore, let
\begin{equation}
\sum\limits_{n = 1}^\infty  {\left\{ {\xi (u,v,\alpha ,\lambda )\left| {{a_n}} \right| + \eta (u,v,\alpha ,\lambda )\left| {{b_n}} \right|} \right\} \le 2} 
\end{equation}
where $\xi (u,v,\alpha ,\lambda ) = {\left[ {1 + \left( {n - 1} \right)\lambda } \right]^v} + \frac{{\left\{ {{{\left[ {1 + \left( {n - 1} \right)\lambda } \right]}^u} - {{\left[ {1 + \left( {n - 1} \right)\lambda } \right]}^v}} \right\}\left( {1 + k} \right)}}{{1 - \alpha }}$\\
$\eta (u,v,\alpha ,\lambda ) = {( - 1)^{v - u}}{\left[ {1 + \left( {n - 1} \right)\lambda } \right]^v} + \frac{{\left\{ {{{\left[ {1 + \left( {n - 1} \right)\lambda } \right]}^u} - {{( - 1)}^{v - u}}{{\left[ {1 + \left( {n - 1} \right)\lambda } \right]}^v}} \right\}\left( {1 + k} \right)}}{{1 - \alpha }}$\\
with $\left|a_1\right|=1$, $0\le \alpha <1, 0\le \alpha <\infty, u \in N = \left\{ {1,2,3,...} \right\}, v \in N_{0} $ and $u>v$. Then $f$ is harmonic univalent, sense-preserving in  $k$-USH $(u,v,\alpha, \lambda)$. The result is sharp also.
\end{theorem}
\begin{proof}
For $\left| {{z_1}} \right| \le \left| {{z_2}} \right| < 1$,
\begin{align*}
\left| {f({z_1}) - f({z_2})} \right| &\ge \left| {h({z_1}) - h({z_2})} \right| - \left| {g({z_1}) - g({z_2})} \right|\\
&\ge \left| {{z_1} - {z_2}} \right|\left( {1 - \sum\limits_{n = 2}^\infty  {n\left| {{a_n}} \right|{{\left| {{z_2}} \right|}^{n - 1}} - \sum\limits_{n = 1}^\infty  {n\left| {{b_n}} \right|{{\left| {{z_2}} \right|}^{n - 1}}} } } \right)\\
&\ge \left| {{z_1} - {z_2}} \right|\left[ {1 - \left| {{z_2}} \right|\left( {\sum\limits_{n = 2}^\infty  {n\left| {{a_n}} \right| + \sum\limits_{n = 1}^\infty  {n\left| {{b_n}} \right|} } } \right)} \right]\\
&\ge \left| {{z_1} - {z_2}} \right|\left[ {1 - \left| {{z_2}} \right|\left( {\sum\limits_{n = 2}^\infty  {\xi (u,v,\alpha ,\lambda )\left| {{a_n}} \right| + \sum\limits_{n = 1}^\infty  {\eta (u,v,\alpha ,\lambda )\left| {{b_n}} \right|} } } \right)} \right]\\
&\ge \left| {{z_1} - {z_2}} \right|\left[ {1 - \left| {{z_2}} \right|} \right]\\
&>0.
\end{align*}
Hence, $f$ is univalent in $U$.\\
$f$ is sense-preserving in $U$, this is because
\begin{align*}
\left| {{h^{'}}(z)} \right| &\ge 1 - \sum\limits_{n = 2}^\infty  {n\left| {{a_n}} \right|{{\left| z \right|}^{n - 1}}} \\
&\ge 1 - \sum\limits_{n = 2}^\infty  {n\left| {{a_n}} \right|}\\
&\ge 1-{\left\{ {{{\left[ {1 + \left( {n - 1} \right)\lambda } \right]}^v} + \frac{{\left\{ {{{\left[ {1 + \left( {n - 1} \right)\lambda } \right]}^u} - {{\left[ {1 + \left( {n - 1} \right)\lambda } \right]}^v}} \right\}\left( {1 + k} \right)}}{{1 - \alpha }}} \right\}\left| {{a_n}} \right|}\\
& \ge {\sum\limits_{n = 1}^\infty  {\left\{ {{{( - 1)}^{v - u}}{{\left[ {1 + \left( {n - 1} \right)\lambda } \right]}^v} + \frac{{\left\{ {{{\left[ {1 + \left( {n - 1} \right)\lambda } \right]}^u} - {{( - 1)}^{v - u}}{{\left[ {1 + \left( {n - 1} \right)\lambda } \right]}^v}} \right\}\left( {1 + k} \right)}}{{1 - \alpha }}} \right\}\left| {{b_n}} \right|} }\\
&\ge {\sum\limits_{n = 1}^\infty  {\left\{ {{{( - 1)}^{v - u}}{{\left[ {1 + \left( {n - 1} \right)\lambda } \right]}^v} + \left\{ {{{\left[ {1 + \left( {n - 1} \right)\lambda } \right]}^u} - {{\left[ {1 + \left( {n - 1} \right)\lambda } \right]}^v}} \right\}\left( {1 + k} \right)} \right\}\left| {{b_n}} \right|} }\\
& \ge \sum\limits_{n = 1}^\infty  {n\,\left| {{b_n}} \right|} \\
& \ge\sum\limits_{n = 1}^\infty  {n\,\left| {{b_n}} \right|{{\left| z \right|}^{n - 1}}}\\
& \ge \left| {{g^{'}}(z)} \right|.
\end{align*}
Now only needs to show that $f$ $\in$ $k$-USH $(u,v,\alpha, \lambda)$.\\
That is 
\begin{align*}
{\mathop{\rm Re}\nolimits} \left\{ {\frac{{\left( {1 + k{e^{i\phi }}} \right){D^u}f(z) - \left( {k{e^{i\phi }} + \alpha } \right){D^v}f(z)}}{{{D^v}f(z)}}} \right\} \ge 0.
\end{align*}
Or,
\begin{align*}
Re \biggl\{{\frac{{\left( {1 + k{e^{i\phi }}} \right)\left( {z + \sum\limits_{n = 2}^\infty  {{{\left[ {1 + \left( {n - 1} \right)\lambda } \right]}^u}\,{a_n}\,{z^n} + {{( - 1)}^u}\sum\limits_{n = 1}^\infty  {{{\left[ {1 + \left( {n - 1} \right)\lambda } \right]}^u}\,{{\overline b }_n}\,{{\overline z }^n}} } } \right)}}{{z + \sum\limits_{n = 2}^\infty  {{{\left[ {1 + \left( {n - 1} \right)\lambda } \right]}^v}\,{a_n}\,{z^n} + {{( - 1)}^u}\sum\limits_{n = 1}^\infty  {{{\left[ {1 + \left( {n - 1} \right)\lambda } \right]}^v}\,{{\overline b }_n}\,{{\overline z }^n}} } }}}\\-{\frac{{\left( {k{e^{i\phi }} + \alpha } \right)\left( {z + \sum\limits_{n = 2}^\infty  {{{\left[ {1 + \left( {n - 1} \right)\lambda } \right]}^v}\,{a_n}\,{z^n} + {{( - 1)}^u}\sum\limits_{n = 1}^\infty  {{{\left[ {1 + \left( {n - 1} \right)\lambda } \right]}^v}\,{{\overline b }_n}\,{{\overline z }^n}} } } \right)}}{{z + \sum\limits_{n = 2}^\infty  {{{\left[ {1 + \left( {n - 1} \right)\lambda } \right]}^v}\,{a_n}\,{z^n} + {{( - 1)}^u}\sum\limits_{n = 1}^\infty  {{{\left[ {1 + \left( {n - 1} \right)\lambda } \right]}^v}\,{{\overline b }_n}\,{{\overline z }^n}} } }}}\biggr\}\ge0.
\end{align*}
\begin{align*}
=Re \biggl\{{\frac{{z\left( {1 - \alpha } \right) + \sum\limits_{n = 2}^\infty  {\left[ {{{\left[ {1 + \left( {n - 1} \right)\lambda } \right]}^u}\left( {1 + k{e^{i\phi }}} \right) - {{\left[ {1 + \left( {n - 1} \right)\lambda } \right]}^v}\left( {k{e^{i\phi }} + \alpha } \right)} \right]{a_n}{z^n}} }}{{z + \sum\limits_{n = 2}^\infty  {{{\left[ {1 + \left( {n - 1} \right)\lambda } \right]}^v}\,{a_n}\,{z^n} + {{( - 1)}^u}\sum\limits_{n = 1}^\infty  {{{\left[ {1 + \left( {n - 1} \right)\lambda } \right]}^v}\,{{\overline b }_n}\,{{\overline z }^n}} } }}}\\
+{\frac{{{{( - 1)}^u}\sum\limits_{n = 1}^\infty  {\left[ {{{\left[ {1 + \left( {n - 1} \right)\lambda } \right]}^u}\left( {1 + k{e^{i\phi }}} \right) - {{( - 1)}^{v - u}}{{\left[ {1 + \left( {n - 1} \right)\lambda } \right]}^v}\left( {k{e^{i\phi }} + \alpha } \right)} \right]} \,{{\overline b }_n}\,{{\overline z }^n}}}{{z + \sum\limits_{n = 2}^\infty  {{{\left[ {1 + \left( {n - 1} \right)\lambda } \right]}^v}\,{a_n}\,{z^n} + {{( - 1)}^u}\sum\limits_{n = 1}^\infty  {{{\left[ {1 + \left( {n - 1} \right)\lambda } \right]}^v}\,{{\overline b }_n}\,{{\overline z }^n}} } }}}
\biggr\} .
\end{align*}
\begin{align*}
=Re \biggl\{{\frac{{\left( {1 - \alpha } \right) + \sum\limits_{n = 2}^\infty  {\left[ {{{\left[ {1 + \left( {n - 1} \right)\lambda } \right]}^u}\left( {1 + k{e^{i\phi }}} \right) - {{\left[ {1 + \left( {n - 1} \right)\lambda } \right]}^v}\left( {k{e^{i\phi }} + \alpha } \right)} \right]{a_n}{z^{n - 1}}} }}{{1 + \sum\limits_{n = 2}^\infty  {{{\left[ {1 + \left( {n - 1} \right)\lambda } \right]}^v}\,{a_n}\,{z^{n - 1}} + {{( - 1)}^u}\sum\limits_{n = 1}^\infty  {{{\left[ {1 + \left( {n - 1} \right)\lambda } \right]}^v}\,{{\overline b }_n}\,{{\overline z }^n}{z^{ - 1}}} } }}}\\
+ {\frac{{{{( - 1)}^u}\sum\limits_{n = 1}^\infty  {\left[ {{{\left[ {1 + \left( {n - 1} \right)\lambda } \right]}^u}\left( {1 + k{e^{i\phi }}} \right) - {{( - 1)}^{v - u}}{{\left[ {1 + \left( {n - 1} \right)\lambda } \right]}^v}\left( {k{e^{i\phi }} + \alpha } \right)} \right]} \,{{\overline b }_n}\,{{\overline z }^n}{z^{ - 1}}}}{{1 + \sum\limits_{n = 2}^\infty  {{{\left[ {1 + \left( {n - 1} \right)\lambda } \right]}^v}\,{a_n}\,{z^{n - 1}} + {{( - 1)}^u}\sum\limits_{n = 1}^\infty  {{{\left[ {1 + \left( {n - 1} \right)\lambda } \right]}^v}\,{{\overline b }_n}\,{{\overline z }^n}{z^{ - 1}}} } }}}\biggr\}.
\end{align*} 
\begin{align*}
{\mathop{\rm Re}\nolimits} \left\{ {\frac{{\left( {1 - \alpha } \right) + A(z)}}{{1 + B(z)}}} \right\} \ge 0.
\end{align*}
For $z = r{e^{i\theta }}$, 
\begin{align*}
A(r\,{e^{i\theta }}) = {\sum\limits_{n = 2}^\infty  {\left[ {{{\left[ {1 + \left( {n - 1} \right)\lambda } \right]}^u}\left( {1 + k{e^{i\phi }}} \right) - {{\left[ {1 + \left( {n - 1} \right)\lambda } \right]}^v}\left( {k{e^{i\phi }} + \alpha } \right)} \right]{a_n}{r^{n - 1}}{e^{\left( {n - 1} \right)\theta i}}} }\\
+{{{( - 1)}^u}\sum\limits_{n = 1}^\infty  {\left[ {{{\left[ {1 + \left( {n - 1} \right)\lambda } \right]}^u}\left( {1 + k{e^{i\phi }}} \right) - {{( - 1)}^{v - u}}{{\left[ {1 + \left( {n - 1} \right)\lambda } \right]}^v}\left( {k{e^{i\phi }} + \alpha } \right)} \right]} \,{{\overline b }_n}\,{r^{n - 1}}{e^{ - \left( {n + 1} \right)\theta i}}}
\end{align*}
\begin{align*}
B(r\,{e^{i\theta }}) ={\sum\limits_{n = 2}^\infty  {{{\left[ {1 + \left( {n - 1} \right)\lambda } \right]}^v}\,{a_n}{r^{n - 1}}{e^{\left( {n - 1} \right)\theta i}}\, + {{( - 1)}^v}\sum\limits_{n = 1}^\infty  {{{\left[ {1 + \left( {n - 1} \right)\lambda } \right]}^v}\,{{\overline b }_n}\,{r^{n - 1}}{e^{ - \left( {n + 1} \right)\theta i}}} } }
\end{align*}
Setting
\begin{align*}
\frac{{\left( {1 - \alpha } \right) + A(z)}}{{1 + B(z)}} = \left( {1 - \alpha } \right)\frac{{1 + w(z)}}{{1 - w(z)}}
\end{align*}
The proof will be completed if it can shown be that $\left| {w(z)} \right| \le r < 1$. Since, 
\begin{align*}
\left| {w(z)} \right| &= \left| {\frac{{A(z) - \left( {1 - \alpha } \right)B(z)}}{{A(z) + \left( {1 - \alpha } \right)B(z) + 2\left( {1 - \alpha } \right)}}} \right|
\end{align*}

\begin{equation*}
=\left|\frac{\splitfrac{\sum\limits_{n = 2}^\infty  {\left[ {{{\left[ {1 {+} \left( {n {-} 1} \right)\lambda } \right]}^u}\left( {1 {+} k{e^{i\phi }}} \right) {-} {{\left[ {1 {+} \left( {n {-} 1} \right)\lambda } \right]}^v}\left[ {\left( {k{e^{i\phi }} {+} \alpha } \right) {-} \left( {1 {-} \alpha } \right)} \right]} \right]} \,{a_n}\,{r^{n {-} 1}}\,{e^{(n {-} 1)\theta i}}}{\splitfrac{{{+}( {-} 1)^u}\sum\limits_{n {=} 1}^\infty  \bigg[{{{\left[ {1 {+} \left( {n {-} 1} \right)\lambda } \right]}^u}\left( {1 {+} k{e^{i\phi }}} \right) {-} {{( {-} 1)}^{v {-} u}}{{\left[ {1 {+} \left( {n {-} 1} \right)\lambda } \right]}^v}\left( {k{e^{i\phi }} {+} \alpha } \right)} }{ {-} \left( {1 {-} \alpha } \right){\left[ {1 {+} \left( {n {-} 1} \right)\lambda } \right]^v}\bigg]{\overline b _n}\,{r^{n {-} 1}}{e^{ {-} \left( {n {+} 1} \right)\theta i}}}}}{\splitfrac{\splitfrac{2(1 {-} \alpha ) {+} \sum\limits_{n {=} 2}^\infty  \bigg[{{{\left[ {1 {+} \left( {n {-} 1} \right)\lambda } \right]}^u}\left( {1{+} k{e^{i\phi }}} \right)}  {+} {\left[ {1 {+} \left( {n {-} 1} \right)\lambda } \right]^v}\left[ {\left( {1 {+} \alpha } \right) {-} \left( {k{e^{i\phi }} {+} \alpha } \right)} \right]\bigg]{a_n}\,{r^{n {-} 1}}\,{e^{(n {-} 1)\theta i}}}{{{+}( {-} 1)^u}\sum\limits_{n {=} 1}^\infty  {{{\left[ {1 {+} \left( {n {-} 1} \right)\lambda } \right]}^u}\left( {1 {+} k{e^{i\phi }}} \right)} {( {-} 1)^{v {-} u}}{\left[ {1 {+} \left( {n {-} 1} \right)\lambda } \right]^v}}}{\left[ {1 {-} 2\alpha  {-} k{e^{i\phi }}} \right]{\overline b _n}\,{r^{n {-} 1}}{e^{ {-} \left( {n {+} 1} \right)\theta i}}}}\right|
\end{equation*}
\begin{equation*}
\le\frac{\splitfrac{{\sum\limits_{n {=} 2}^\infty  {\left[ {{{\left[ {1 {+} (n {-} 1)\lambda } \right]}^u} {-} {{\left[ {1 {+} (n {-} 1)\lambda } \right]}^v}\left( {1 {+} k} \right)} \right]\left| {{a_n}} \right|\,{r^{n {-} 1}}} }}{{+}\sum\limits_{n {=} 1}^\infty  {\left[ {{{\left[ {1 {+} (n {-} 1)\lambda } \right]}^u} {-} {{( {-} 1)}^{v {-} u}}{{\left[ {1 {+} (n {-} 1)\lambda } \right]}^v}\left( {1 {+} k} \right)} \right]\left| {{b_n}} \right|\,{r^{n {-} 1}}} }}{\splitfrac{\splitfrac{2(1 {-} \alpha ) {-} \sum\limits_{n {=} 2}^\infty  {\left[ {\left( {{{\left[ {1 {+} (n {-} 1)\lambda } \right]}^u} {-} {{\left[ {1 {+} (n {-} 1)\lambda } \right]}^v}} \right)k {+} {{\left[ {1 {+} (n {-} 1)\lambda } \right]}^u} {+} \left( {1 {-} 2\alpha } \right){{\left[ {1 {+} (n {-} 1)\lambda } \right]}^v}} \right]\left| {{a_n}} \right|\,{r^{n {-} 1}}}}{{-}\sum\limits_{n {=} 1}^\infty  {\bigg[\left( {{{\left[ {1 {+} (n {-} 1)\lambda } \right]}^u} {-} {{( {-} 1)}^{v {-} u}}{{\left[ {1 {+} (n {-} 1)\lambda } \right]}^v}} \right)k {+} {{\left[ {1 {+}(n {-} 1)\lambda } \right]}^u}}}}{{+} {( {-} 1)^{v {-} u}}\left( {1 {-} 2\alpha } \right){\left[ {1 {+} (n {-} 1)\lambda } \right]^v}\bigg]\left| {{b_n}} \right|\,{r^{n {-} 1}}}}
\end{equation*}
\begin{equation*}
\le \dfrac{\splitdfrac{\sum\limits_{n = 2}^\infty  {\left( {{{\left[ {1 + (n - 1)\lambda } \right]}^u} - {{\left[ {1 + (n - 1)\lambda } \right]}^v}} \right)(1 + k)\left| {{a_n}} \right|}}{+\sum\limits_{n = 1}^\infty  {\left( {{{\left[ {1 + (n - 1)\lambda } \right]}^u} - {{( - 1)}^{v - u}}{{\left[ {1 + (n - 1)\lambda } \right]}^v}} \right)(1 + k)\left| {{b_n}} \right|} }}{2(1 - \alpha ) - \sum\limits_{n = 2}^\infty  {C(u,v,\alpha ,\lambda )} \left| {{a_n}} \right| + \sum\limits_{n = 1}^\infty  {D(u,v,\alpha ,\lambda )} \left| {{b_n}} \right|}
\end{equation*}
where
\begin{equation*}
C(u,v,\alpha ,\lambda ) = \left( {{{\left[ {1 + (n - 1)\lambda } \right]}^u} - {{\left[ {1 + (n - 1)\lambda } \right]}^v}} \right)k + {\left[ {1 + (n - 1)\lambda } \right]^u} + (1 - 2\alpha ){\left[ {1 + (n - 1)\lambda } \right]^v}
\end{equation*}
\begin{equation*}
D(u,v,\alpha ,\lambda ) {=} \left( {{{\left[ {1 {+} (n {-} 1)\lambda } \right]}^u} {-} {{( {-} 1)}^{v {} u}}{{\left[ {1 + (n - 1)\lambda {-}} \right]}^v}} \right)k {+} {\left[ {1 {+} (n {-} 1)\lambda } \right]^u} {+} {( {-} 1)^{v {-} u}}(1 {-} 2\alpha ){\left[ {1 {+} (n {-} 1)\lambda } \right]^v}
\end{equation*}
\begin{equation*}
\left| {w(z)} \right| \le \dfrac{\splitdfrac{\sum\limits_{n = 2}^\infty  {\left( {{{\left[ {1 + (n - 1)\lambda } \right]}^u} - {{\left[ {1 + (n - 1)\lambda } \right]}^v}} \right)(1 + k)\left| {{a_n}} \right|} }{+\sum\limits_{n = 1}^\infty  {\left( {{{\left[ {1 + (n - 1)\lambda } \right]}^u} - {{\left[ {1 + (n - 1)\lambda } \right]}^v}} \right)(1 + k)\left| {{b_n}} \right|}}}{4(1 - \alpha ) - \sum\limits_{n = 1}^\infty  {\left\{ {C(u,v,\alpha ,\lambda )\left| {{a_n}} \right| + D(u,v,\alpha ,\lambda )\left| {{b_n}} \right|} \right\}}}\le 1.
\end{equation*}
which proves the theorem (2.1).\\
The harmonic univalent function:
\begin{equation}
f(z) = z + \sum\limits_{n = 2}^\infty  {\frac{1}{{\xi (u,v,\alpha ,\lambda )}}} \,{x_n}\,{z^n} + \sum\limits_{n = 1}^\infty  {\frac{1}{{\eta (u,v,\alpha ,\lambda )}}} \,{y_n}\,{\overline z ^n},
\end{equation}
where $u \in N,\,\,\,v \in {N_0}$, $u>v$ and $\sum\limits_{n = 2}^\infty  {{x_n} + } \sum\limits_{n = 1}^\infty  {{y_n} = 1}$, shows that the coefficient bound given by equation (2.1) is sharp. The functions of the form (2.2) are in $k$-USH $(u,v,\alpha, \lambda)$ because
\begin{equation*}
\sum\limits_{n = 1}^\infty  {\left\{ {\xi (u,v,\alpha ,\lambda )\left| {{a_n}} \right| + \eta (u,v,\alpha ,\lambda )\left| {{b_n}} \right|} \right\}}  = 1 + \sum\limits_{n = 2}^\infty  {{x_n} + } \sum\limits_{n = 1}^\infty  {{y_n} = 2}.
\end{equation*}
\end{proof}
For $k=1, u=v+1, \lambda=1$, the following corollary is obtained. 
\begin{corollary}\cite{B8}
Let $f=h+\overline{g} \in SH$ given by equation (1.1) and if 
\begin{equation*}
\sum\limits_{n = 1}^\infty  {{n^v}\left[ {\left( {2n - 1 - \alpha } \right)\left| {{a_n}} \right| + \left( {2n + 1 + \alpha } \right)\left| {{b_n}} \right|} \right]}  \le 2\left( {1 - \alpha } \right)
\end{equation*}
where $\left|a_1\right| = 1$ and $0 \le \alpha  < 1$, then f is harmonic univalent, sense-preserving in $U$ and $f \in RS_{H} (v, \alpha)$.
\end{corollary}
Taking $k=0, u=1, v=0,  \lambda =1$, the following corollary is obtained.
\begin{corollary}\cite{B2}
Let $f=h+\overline{g} \in SH$ given by equation (1.1) and if
\begin{equation*}
\sum\limits_{n = 1}^\infty  {\left[ {\frac{{\left( {n - \alpha } \right)}}{{\left( {1 - \alpha } \right)}}\left| {{a_n}} \right| + \frac{{\left( {n + \alpha } \right)}}{{\left( {1 - \alpha } \right)}}\left| {{b_n}} \right|} \right]}  \le 2,
\end{equation*}
where $\left|a_1\right| = 1$ and $0 \le \alpha  < 1$, then $f$ is harmonic univalent, sense-preserving in $U$ and $f \in SH (\alpha)$.
\end{corollary}
On taking $k=0, u=2, v=1, \lambda=1$ in Theorem 2.1 following result of Jahangiri [3] is obtained.
\begin{corollary}\cite{B3}
Let $f=h+\overline{g} \in SH$ given by equation (1.1) and if 
\begin{equation*}
\sum\limits_{n = 1}^\infty  {\left[ {\frac{{n\left( {n - \alpha } \right)}}{{\left( {1 - \alpha } \right)}}\left| {{a_n}} \right| + \frac{{n\left( {n + \alpha } \right)}}{{\left( {1 - \alpha } \right)}}\left| {{b_n}} \right|} \right]}  \le 2
\end{equation*}
where $\left|a_1\right| = 1$ and $0 \le \alpha  < 1$, then $f$ is harmonic univalent, sense-preserving in $U$ and $f \in KH (\alpha)$.
\end{corollary}
Putting $k=1, u=1, v=0, \lambda=1$, in Theorem 2.1 following result of Rosy et al. [6] is obtained.
\begin{corollary}\cite{B7}
Let $f=h+\overline{g} \in SH$ given by equation (1.1) and if
\begin{equation*}
\sum\limits_{n = 1}^\infty  {\left[ {\left( {2n - 1 - \alpha } \right)\left| {{a_n}} \right| + \left( {2n + 1 + \alpha } \right)\left| {{b_n}} \right|} \right]}  \le 2\left( {1 - \alpha } \right).
\end{equation*}
where $\left|{a_1}\right| = 1$ and $0 \le \alpha  < 1$, then $f$ is harmonic univalent, sense-preserving in $U$ and $f \in GH (\alpha)$.
\end{corollary}
Putting $u=v+1, \lambda=1$, in Theorem 2.1 following result of Khan \cite{B5} is obtained.
\begin{corollary}\cite{B5}
Let $f=h+\overline{g} \in SH$ given by equation (1.1) and if
\begin{equation*}
\sum\limits_{n = 1}^\infty  {{n^v}\left[ {\left( {n + nk - k - \alpha } \right)\left| {{a_n}} \right| + \left( {n + nk + k + \alpha } \right)\left| {{b_n}} \right|} \right]}  \le 2\left( {1 - \alpha } \right).
\end{equation*}
where $\left|a_1\right| = 1$ and $0 \le \alpha  < 1$, then $f$ is harmonic univalent, sense-preserving in $U$ and $f \in k-USH (v+1,v,\alpha)$.
\end{corollary}
Substituting $u=2, v=1, \lambda=1$, the following corollary is obtained.
\begin{corollary}\cite{B6}
Let $f=h+\overline{g} \in SH$ given by equation (1.1) and if
\begin{equation*}
\sum\limits_{n = 1}^\infty  {n\left[ {\left( {n + nk - k - \alpha } \right)\left| {{a_n}} \right| + \left( {n + nk + k + \alpha } \right)\left| {{b_n}} \right|} \right]}  \le 2\left( {1 - \alpha } \right).
\end{equation*}
where $\left|a_1\right| = 1$ and $0 \le \alpha  < 1$, then $f$ is harmonic univalent, sense-preserving in $U$ and $f \in k-HCV (\alpha)$.
\end{corollary} 
Further on substituting $u=1, v=0, \lambda=1$, the following corollary is obtained.
\begin{corollary}\cite{B5}
Let $f=h+\overline{g} \in SH$ given by equation (1.1) and if
\begin{equation*}
\sum\limits_{n = 1}^\infty  {\left[ {\left( {n + nk - k - \alpha } \right)\left| {{a_n}} \right| + \left( {n + nk + k + \alpha } \right)\left| {{b_n}} \right|} \right]}  \le 2\left( {1 - \alpha } \right).
\end{equation*}
where $\left|a_1\right| = 1$ and $0 \le \alpha  < 1$, then $f$ is harmonic univalent, sense-preserving in $U$ and $f \in k-USH (\alpha)$.
\end{corollary} 
On putting $\alpha=0, \lambda=1 $. following corollary is obtained.
\begin{corollary}\cite{B5}
Let $f=h+\overline{g} \in SH$ given by equation (1.1). Furthermore, let
\begin{align*}
\sum\limits_{n = 2}^\infty  {\frac{{n\left( {\left| {{a_n}} \right| + \left| {{b_n}} \right|} \right)}}{{\left[ {1 - \left\{ {1 + k\left( {1 - {{( - 1)}^{v - u}}} \right)\left| {{b_1}} \right|} \right\}} \right]}}}  \le 1
\end{align*}
with $0 \le k < \infty, u \in N = \left\{ {1,2,...} \right\}$ and $u>v$, then $f$ is harmonic univalent, sense-preserving in $U$ and $f \in k-USH(u,v,0)$. 
\end{corollary}
\begin{theorem}
(Coefficient inequality for $k$-GTH $(u,v,\alpha,\lambda))$\\
Let ${f_u} = h + {\overline g _u}$ where $h$ and $g$ be given by equation (1.5).Then
${f_u} \in  k-UTH (u,v,\alpha,\lambda)$ if and only if 
\begin{equation}
\sum\limits_{n = 1}^\infty  {\left\{ {\xi (u,v,\alpha ,\lambda )\left| {{a_n}} \right| + \eta (u,v,\alpha ,\lambda )\left| {{b_n}} \right|} \right\}}  \le 2.
\end{equation}
where $\left| {{a_1}} \right| = 1, 0 \le \alpha  < 1, u \in N, v \in {N_0}, u > v$.
\end{theorem}
\begin{proof}
Since $k$-GTH $(u,v,\alpha,\lambda))$ is subclass of $k$-GSH $(u,v,\alpha,\lambda))$,\\
it only needs to prove the “only if” part of the theorem.\\
For functions ${f_u}$ of the form (1.5), the condition
\begin{align*}
{\mathop{\rm Re}\nolimits} \left\{ {\left( {1 + k{e^{i\phi }}} \right)\frac{{{D^u}f(z)}}{{{D^v}f(z)}} - k{e^{i\phi }}} \right\} \ge \alpha 
\end{align*}
is equivalent to
\begin{equation*}
{\mathop{\rm Re}\nolimits} \left\{ {\dfrac{\splitdfrac{\left( {z - \sum\limits_{n = 2}^\infty  {{{\left[ {1 + (n - 1)\lambda } \right]}^u}\left| {{a_n}} \right|{z^n} + {{\left( { - 1} \right)}^{2u - 1}}\sum\limits_{n = 1}^\infty  {{{\left[ {1 + (n - 1)\lambda } \right]}^u}\left| {{b_n}} \right|{{\overline z }^n}} } } \right)\left( {1 + k{e^{i\phi }}} \right)}{-\left( {z - \sum\limits_{n = 2}^\infty  {{{\left[ {1 + (n - 1)\lambda } \right]}^v}\left| {{a_n}} \right|{z^n} + {{\left( { - 1} \right)}^{u + v - 1}}\sum\limits_{n = 1}^\infty  {}{{\left[ {1 + (n - 1)\lambda } \right]}^v}\left| {{b_n}} \right|{{\overline z }^n} } } \right)\left( {\alpha  + k{e^{i\phi }}} \right)}}{z - \sum\limits_{n = 2}^\infty  {{{\left[ {1 + (n - 1)\lambda } \right]}^v}\left| {{a_n}} \right|{z^n} + {{\left( { - 1} \right)}^{u + v - 1}}\sum\limits_{n = 1}^\infty  {}{{\left[ {1 + (n - 1)\lambda } \right]}^v}\left| {{b_n}} \right|{{\overline z }^n} }}} \right\}\ge0
\end{equation*}
or
\begin{equation*}
{\mathop{\rm Re}\nolimits} \left\{ {\dfrac{\splitdfrac{\splitdfrac{z {-} \sum\limits_{n {=} 2}^\infty  {{{\left[ {1 {+} (n {-} 1)\lambda } \right]}^u}\left| {{a_n}} \right|{z^n} {+} {{\left( { {-} 1} \right)}^{2u {-} 1}}\sum\limits_{n {=} 1}^\infty  {{{\left[ {1 {+} (n {-} 1)\lambda } \right]}^u}} } \left| {{b_n}} \right|{\overline z ^n}}{{-}k{e^{i\phi }}\sum\limits_{n {=} 2}^\infty  {{{\left[ {1 {+} (n {-} 1)\lambda } \right]}^u}\left| {{a_n}} \right|{z^n} {+} } k{e^{i\phi }}{\left( { {-} 1} \right)^{2u {-} 1}}{\sum\limits_{n {=} 1}^\infty  {{{\left[ {1 {+} (n {-} 1)\lambda } \right]}^v}\left| {{b_n}} \right|{{\overline z }^n}} }}}{{+}\sum\limits_{n {=} 2}^\infty  {\left( {\alpha  {+} k{e^{i\phi }}} \right){{\left[ {1 {+} (n {-} 1)\lambda } \right]}^v}} \left| {{a_n}} \right|{z^n}{{-}\left( { {-} 1} \right)^{u {+} v {-} 1}}\sum\limits_{n {=} 1}^\infty  {{{\left[ {1 {+} (n {-} 1)\lambda } \right]}^v}\left| {{b_n}} \right|{{\overline z }^n}(1 {+} k{e^{i\phi }}) {-} \alpha z}}}{z {-} \sum\limits_{n {=} 2}^\infty  {{{\left[ {1 {+} (n {-} 1)\lambda } \right]}^v}\left| {{a_n}} \right|{z^n} {+} {{\left( { {-} 1} \right)}^{u {+} v {-} 1}}\sum\limits_{n {=} 1}^\infty  {}{{\left[ {1 {+} (n {-} 1)\lambda } \right]}^v \left| {{b_n}} \right|{{\overline z }^n}} } }} \right\}\ge0
\end{equation*}
or
\begin{equation*}
{\mathop{\rm Re}\nolimits} \left\{ {\dfrac{\splitdfrac{\splitdfrac{z\left( {1 - \alpha } \right) - \sum\limits_{n = 2}^\infty  {{{\left[ {1 + (n - 1)\lambda } \right]}^u} + k{e^{i\phi }}{{\left[ {1 + (n - 1)\lambda } \right]}^u} - k{e^{i\phi }}{{\left[ {1 + (n - 1)\lambda } \right]}^v}} }{{ - \alpha {{\left[ {1 + (n - 1)\lambda } \right]}^v}\left| {{a_n}} \right|{z^n}}+ {\left( { - 1} \right)^{2u - 1}}\sum\limits_{n = 1}^\infty  {\left( {{{\left[ {1 + (n - 1)\lambda } \right]}^u} + k{e^{i\phi }}{{\left[ {1 + (n - 1)\lambda } \right]}^u}} \right)}}}{- {\left( { - 1} \right)^{v - u}}\left( {\alpha {{\left[ {1 + (n - 1)\lambda } \right]}^v}k{e^{i\phi }}{{\left[ {1 + (n - 1)\lambda } \right]}^v}} \right)\left| {{b_n}} \right|{\overline z ^n}}}{{z - \sum\limits_{n = 2}^\infty  {{{\left[ {1 + (n - 1)\lambda } \right]}^v}\left| {{a_n}} \right|{z^n} + {{\left( { - 1} \right)}^{u + v - 1}}\sum\limits_{n = 1}^\infty  {}{{\left[ {1 + (n - 1)\lambda } \right]}^v \left| {{b_n}} \right|{{\overline z }^n}} } }}} \right\}\ge0.
\end{equation*}
The above condition must hold for all values of $z$. Choosing $z$ on the positive real axis as $z\rightarrow 1$, the above inequality becomes
\begin{equation*}
{\mathop{\rm Re}\nolimits} \left\{ {\dfrac{\splitdfrac{\left( {1 {-} \alpha } \right) {-} \sum\limits_{n {=} 2}^\infty  {\left( {{{\left[ {1 {+} (n {-} 1)\lambda } \right]}^u} {+} k{e^{i\phi }}{{\left[ {1 {+} (n {-} 1)\lambda } \right]}^u} {-} k{e^{i\phi }}{{\left[ {1 {+} (n {-} 1)\lambda } \right]}^v} {-} \alpha {{\left[ {1 {+} (n {-} 1)\lambda } \right]}^v}} \right)} }{\times\left| {{a_n}} \right|r^{n-1}{+} {\left( { {-} 1} \right)^{2u {-} 1}}\sum\limits_{n {=} 1}^\infty  {\left( {\left( {1 {+} k{e^{i\phi }}} \right){{\left[ {1 {+} (n {-} 1)\lambda } \right]}^u} {-} {{\left( { {-} 1} \right)}^{v {-} u}}\left( {k{e^{i\phi }} {-} \alpha } \right){{\left[ {1 {+} (n {-} 1)\lambda } \right]}^v}} \right)} \left| {{b_n}} \right|r^{n-1}}}{1 {-} \sum\limits_{n {=} 2}^\infty  {{{\left[ {1 {+} (n {-} 1)\lambda } \right]}^v}} \left| {{a_n}} \right| {-} {\left( { {-} 1} \right)^{v {+} u}}\sum\limits_{n {=} 2}^\infty  {{{\left[ {1 {+} (n {-} 1)\lambda } \right]}^v}} \left| {{b_n}} \right|}} \right\}\ge0.
\end{equation*}
Since,  
\begin{equation*}
{\mathop{\rm Re}\nolimits} \left\{ {{e^{i\phi }}} \right\} \le \left| {{e^{i\phi }}} \right| = 1
\end{equation*}
the above inequality reduces to
\begin{align*}
\sum\limits_{n = 2}^\infty  {\left( {{{\left[ {1 + (n - 1)\lambda } \right]}^u}(1 + k) - {{\left[ {1 + (n - 1)\lambda } \right]}^v}(k + \alpha )} \right)} \left| {{a_n}} \right| \\
+ \sum\limits_{n = 1}^\infty  {\left( {{{\left[ {1 + (n - 1)\lambda } \right]}^u}(1 + k) - {{\left( { - 1} \right)}^{v - u}}{{\left[ {1 + (n - 1)\lambda } \right]}^v}(k + \alpha )} \right)} \left| {{b_n}} \right| \le 1 - \alpha.
\end{align*}
Thus, the theorem is proved.
\end{proof}
\pagebreak

Taking $k=0, u=1, v=0, \lambda=1$, the following result is obtained.
\begin{corollary}\cite{B2}
Let $f_u=h+\overline{g_u} \in SH$ given by equation (1.5). Then $f_u \in TH(\alpha)$ if and only if 
\begin{equation*}
\sum\limits_{n = 1}^\infty  {\left[ {\frac{{n - \alpha }}{{1 - \alpha }}\left| {{a_n}} \right| + \frac{{n + \alpha }}{{1 - \alpha }}\left| {{b_n}} \right|} \right]}  \le 2.
\end{equation*}
where $\left| {{a_1}} \right| = 1$ and  $0 \le \alpha  < 1$.
\end{corollary}
Taking $k=0, u=2, v=1, \lambda=1$, the following result is obtained.
\begin{corollary}\cite{B3}
Let $f_u=h+\overline{g_u} \in SH$ given by equation (1.5). Then $f_u \in KH(\alpha)$ if and only if 
\begin{equation*}
\sum\limits_{n = 1}^\infty  {\left[ {\frac{{n\left( {n - \alpha } \right)}}{{1 - \alpha }}\left| {{a_n}} \right| + \frac{{n\left( {n + \alpha } \right)}}{{1 - \alpha }}\left| {{b_n}} \right|} \right]}  \le 2.
\end{equation*}
where $\left| {{a_1}} \right| = 1$ and  $0 \le \alpha  < 1$.
\end{corollary}
Taking $k=1, u=v+1, \lambda=1$, the following result is obtained.
\begin{corollary}\cite{B8}
Let $f_u=h+\overline{g_u} \in SH$ given by equation (1.5). Then $f_u \in RS_{H}(v,\alpha)$ if and only if 
\begin{equation*}
\sum\limits_{n = 1}^\infty  {{n^v}} \left[ {\left( {2n - 1 - \alpha } \right)\left| {{a_n}} \right| + \left( {2n + 1 + \alpha } \right)\left| {{b_n}} \right|} \right] \le 2\left( {1 - \alpha } \right).
\end{equation*}
where $\left| {{a_1}} \right| = 1$, $0 \le \alpha  < 1$ and $v \in N_{0}$.
\end{corollary}
Taking $k=1, u=1, v=0, \lambda=1$, the following result is obtained.
\begin{corollary}\cite{B7}
Let $f_u=h+\overline{g_u} \in SH$ given by equation (1.5). Then $f_u \in G_{H}(\alpha)$ if and only if 
\begin{equation*}
\sum\limits_{n = 1}^\infty  {\left[ {\left( {2n - 1 - \alpha } \right)\left| {{a_n}} \right| + \left( {2n + 1 + \alpha } \right)\left| {{b_n}} \right|} \right] \le 2\left( {1 - \alpha } \right)}.
\end{equation*}
where $\left| {{a_1}} \right| = 1$ and  $0 \le \alpha  < 1$.
\end{corollary}
On putting $u=v+1, \lambda=1$, the following result is obtained.
\begin{corollary}\cite{B5}
Let $f_u=h+\overline{g_u} \in SH$ given by equation (1.5). Then $f_u \in k-UTH(v+1, v,\alpha)$ if and only if 
\begin{equation*}
\sum\limits_{n = 1}^\infty  {{n^v}\left[ {\left( {n + nk - 1 - \alpha } \right)\left| {{a_n}} \right| + \left( {n + nk + 1 + \alpha } \right)\left| {{b_n}} \right|} \right] \le 2\left( {1 - \alpha } \right)}.
\end{equation*}
where $\left| {{a_1}} \right| = 1$, $0 \le \alpha  < 1$ and $v \in N_{0}$.
\end{corollary}
On putting $u=2, v=1, \lambda=1$, the following result is obtained.
\begin{corollary}\cite{B6}
Let $f_u=h+\overline{g_u} \in SH$ given by equation (1.5). Then $f_u \in k-THCV(\alpha)$ if and only if 
\begin{equation*}
\sum\limits_{n = 1}^\infty  {n\left[ {\left( {n + nk - 1 - \alpha } \right)\left| {{a_n}} \right| + \left( {n + nk + 1 + \alpha } \right)\left| {{b_n}} \right|} \right] \le 2\left( {1 - \alpha } \right)}.
\end{equation*}
where $\left| {{a_1}} \right| = 1$ and  $0 \le \alpha  < 1$.
\end{corollary}
On putting $u=1, v=0, \lambda=1$, the following result is obtained.
\begin{corollary}\cite{B5}
Let $f_u=h+\overline{g_u} \in SH$ given by equation (1.5). Then $f_u \in k-UTH(\alpha)$ if and only if 
\begin{equation*}
\sum\limits_{n = 1}^\infty  {\left[ {\left( {n + nk - 1 - \alpha } \right)\left| {{a_n}} \right| + \left( {n + nk + 1 + \alpha } \right)\left| {{b_n}} \right|} \right] \le 2\left( {1 - \alpha } \right)}.
\end{equation*}
where $\left| {{a_1}} \right| = 1$ and  $0 \le \alpha  < 1$.
\end{corollary}
\begin{theorem}(Extreme Points)\\
Let $f_u=h+\overline{g_u} \in SH$ given by equation (1.5). Then $f_u \in k-UTH(u, v, \alpha, \lambda)$ if and only if 
\begin{equation}
{f_u} = \sum\limits_{n = 1}^\infty  {\left[ {x{}_n{P_n}(z) + y{}_n{Q_n}(z)} \right]} 
\end{equation}
where, ${P_1}(z) = z$, ${P_n}(z) = z - \frac{1}{{\xi (u,v,\alpha ,\lambda )}}{z^n}$, \quad $\left( {n = 2,3,...} \right)$ and ${Q_n}(z) = z + {\left( { - 1} \right)^{u - 1}}\frac{1}{{\eta (u,v,\alpha ,\lambda )}}{z^{ - n}}$, \quad $\left( {n =1, 2,3,...} \right)$, ${x_n} \ge 0, {y_n} \ge 0, {x_1} = 1 - \sum\limits_{n = 2}^\infty  {{x_n}}  - \sum\limits_{n = 1}^\infty  {{y_n}}$.\\
In particular, the extreme points of $k-GTH(u, v, \alpha, \lambda)$ are $\left\{ {{P_n}} \right\}$ and $\left\{ {{Q_n}} \right\}$.
\end{theorem}
\begin{proof}
For the functions $f_u$ of the form (1.5)
\begin{align*}
{f_u} &= \sum\limits_{n = 1}^\infty  {\left[ {x{}_n{P_n}(z) + y{}_n{Q_n}(z)} \right]}\\
& = {x_1}z + \sum\limits_{n = 2}^\infty  {x{}_n} \left[ {z - \frac{1}{{\xi (u,v,\alpha ,\lambda )}}{z^n}} \right] + y{}_n\left[ {z + {{\left( { - 1} \right)}^{u - 1}}\frac{1}{{\eta (u,v,\alpha ,\lambda )}}{z^{ - n}}} \right]\\
&= z\left[ {{x_1} + \sum\limits_{n = 2}^\infty  {{x_n}}  + \sum\limits_{n = 1}^\infty  {{y_n}} } \right] - \sum\limits_{n = 2}^\infty  {\frac{{{x_n}}}{{\xi (u,v,\alpha ,\lambda )}}} {z^n} + {\left( { - 1} \right)^{u - 1}}\sum\limits_{n = 1}^\infty  {\frac{{{y_n}}}{{\eta (u,v,\alpha ,\lambda )}}} {z^{ - n}}\\
& = z - \sum\limits_{n = 2}^\infty  {\frac{{{x_n}}}{{\xi (u,v,\alpha ,\lambda )}}} {z^n} + {\left( { - 1} \right)^{u - 1}}\sum\limits_{n = 1}^\infty  {\frac{{{y_n}}}{{\eta (u,v,\alpha ,\lambda )}}} {z^{ - n}} \in k-UTH(u, v, \alpha, \lambda).
\end{align*}
Since,
\begin{equation*}
\sum\limits_{n = 2}^\infty  {\xi (u,v,\alpha ,\lambda )\left( {\frac{{{x_n}}}{{\xi (u,v,\alpha ,\lambda )}}} \right)}  + \sum\limits_{n = 1}^\infty  {\eta (u,v,\alpha ,\lambda )\left( {\frac{{{y_n}}}{{\eta (u,v,\alpha ,\lambda )}}} \right)}  = \sum\limits_{n = 2}^\infty  {{x_n}}  + \sum\limits_{n = 1}^\infty  {{y_n} = 1 - {x_1} \le 1}.
\end{equation*}
and we have $f_u \in clco (k-UTH(u, v, \alpha, \lambda))$.\\
Conversely, suppose $f_u \in clco (k-UTH(u, v, \alpha, \lambda))$.\\
 Letting ${x_1} = 1 - \sum\limits_{n = 2}^\infty  {{x_n}}  - \sum\limits_{n = 1}^\infty  {{y_n}}$. Set ${x_n} = \xi (u,v,\alpha ,\lambda )\left| {{a_n}} \right|$, (n=2,3,...) and\\
  ${y_n} = \eta (u,v,\alpha ,\lambda )\left| {{b_n}} \right|$, (n=1,2,3,...). The required representation is obtained.\\
 Since, 
 \begin{align*}
 {f_u}(z) &= \sum\limits_{n = 2}^\infty  {\left| {{a_n}} \right|{z^n}}  + {\left( { - 1} \right)^{u - 1}}\sum\limits_{n = 1}^\infty  {\left| {{b_n}} \right|{z^{ - n}}} \\
 & = z - \sum\limits_{n = 2}^\infty  {\frac{{{x_n}}}{{\xi (u,v,\alpha ,\lambda )}}} {z^n} + {\left( { - 1} \right)^{u - 1}}\sum\limits_{n = 1}^\infty  {\frac{{{y_n}}}{{\eta (u,v,\alpha ,\lambda )}}} {z^{ - n}} \\
 & = z - \sum\limits_{n = 2}^\infty  {\left( {z - {P_n}(z)} \right){x_n}}  + \sum\limits_{n = 1}^\infty  {\left( {z - {Q_n}(z)} \right){y_n}}\\
 & = z\left( {1 - \sum\limits_{n = 2}^\infty  {{x_n}}  - \sum\limits_{n = 1}^\infty  {{y_n}} } \right) + \sum\limits_{n = 2}^\infty  {{P_n}(z){x_n}}  + \sum\limits_{n = 1}^\infty  {{Q_n}(z){y_n}}\\
 & = z{x_1} + \sum\limits_{n = 2}^\infty  {{P_n}(z){x_n}}  + \sum\limits_{n = 1}^\infty  {{Q_n}(z){y_n}}\\
 & = \sum\limits_{n = 1}^\infty  {\left( {{P_n}(z){x_n}+{Q_n}(z){y_n}} \right)}.
 \end{align*}
\end{proof}
\begin{theorem} (Distortion Bounds)\\
Let $f_u \in k-UTH(u, v, \alpha, \lambda)$. Then for $\left| z \right| = r < 1$ 
\begin{align*}
\left| {{f_u}(z)} \right| \le \left( {1 + \left| {{b_1}} \right|r} \right) + \left\{ {\sigma (u,v,\alpha ,\lambda ) - \tau (u,v,\alpha ,\lambda )\left| {{b_1}} \right|} \right\}{r^2}
\end{align*}
and
\begin{align*}
\left| {{f_u}(z)} \right| \ge \left( {1 - \left| {{b_1}} \right|r} \right) - \left\{ {\sigma (u,v,\alpha ,\lambda ) - \tau (u,v,\alpha ,\lambda )\left| {{b_1}} \right|} \right\}{r^2},
\end{align*}
where
\begin{equation*}
 \sigma (u,v,\alpha ,\lambda ) = \frac{{1 - \alpha }}{{{{\left( {1 + \lambda } \right)}^u}\left( {1 + k} \right) - {{\left( {1 + \lambda } \right)}^v}\left(k+ {\alpha  } \right)}}, \quad \tau (u,v,\alpha ,\lambda ) = \frac{{\left( {1 + k} \right) - {{\left( { - 1} \right)}^{v - u}}\left( {k + \alpha } \right)}}{{{{\left( {1 + \lambda } \right)}^u}\left( {1 + k} \right) - {{\left( {1 + \lambda } \right)}^v}\left( {k + \alpha } \right)}}.
\end{equation*}
\end{theorem}
\begin{proof}
We will prove only the right hand side inequality, the proof for the left hand inequality can be done using similar arguments.\\
Let $f_u \in k-UTH(u, v, \alpha, \lambda)$. Taking the absolute value of $f_u$
\begin{align*}
\left| {{f_u}(z)} \right|& = \left| {z - \sum\limits_{n = 2}^\infty  {\left| {{a_n}} \right|{z^n} + {{\left( { - 1} \right)}^{u - 1}}\sum\limits_{n = 1}^\infty  {\left| {{b_n}} \right|{{\overline z }^n}} } } \right|\\
& \le r + \sum\limits_{n = 2}^\infty  {\left| {{a_n}} \right|{r^n} + \sum\limits_{n = 1}^\infty  {\left| {{b_n}} \right|{r^n}} } \\
& \le r + \left| {{b_1}} \right|r + \sum\limits_{n = 1}^\infty  {\left( {\left| {{a_n}} \right| + \left| {{b_n}} \right|} \right){r^n}} \\
&\le \left( {1 + \left| {{b_1}} \right|} \right)r + \sigma (u,v,\alpha ,\lambda ){\sum\limits_{n = 1}^\infty  {\left[ {\xi (u,v,\alpha ,\lambda )\left| {{a_n}} \right| + \eta (u,v,\alpha ,\lambda )\left| {{b_n}} \right|} \right]r^2}} \\
&\le \left( {1 + \left| {{b_1}} \right|r} \right) + \left\{ {\sigma (u,v,\alpha ,\lambda ) - \tau (u,v,\alpha ,\lambda )\left| {{b_1}} \right|} \right\}{r^2}.
\end{align*}
\end{proof}
\begin{theorem}(Convolution Condition)\\
Let $f_u$, $g_u \in k-UTH(u, v, \alpha, \lambda)$, then the convolution	
\begin{equation*}
\left( {{f_u} * {g_u}} \right)\left( z \right) = {f_u}(z) * {g_u}(z) = z - \sum\limits_{n = 2}^\infty  {\left| {{a_n}{c_n}} \right|\,} {z^n} + {\left( { - 1} \right)^{u - 1}}\sum\limits_{n = 1}^\infty  {\left| {{b_n}{d_n}} \right|\,} {\overline z ^n} \in k-GTH(u, v, \alpha, \lambda).
\end{equation*}
where ${f_u}(z) = z - \sum\limits_{n = 2}^\infty  {\left| {{a_n}} \right|\,} {z^n} + {\left( { - 1} \right)^{u - 1}}\sum\limits_{n = 1}^\infty  {\left| {{b_n}} \right|\,} {\overline z ^n}$ and\\ ${g_u}(z) = z - \sum\limits_{n = 2}^\infty  {\left| {{c_n}} \right|\,} {z^n} + {\left( { - 1} \right)^{u - 1}}\sum\limits_{n = 1}^\infty  {\left| {{d_n}} \right|\,} {\overline z ^n}$.
\end{theorem}
\begin{proof}
For $f_u \in k-UTH(u, v, \alpha, \lambda)$, $\left| {{c_n}} \right| \le 1$,  $\left| {{d_n}} \right| \le 1$\\
now, for the convolution function ${{f_u} * {g_u}}$.\\
Consider, 
\begin{align*}
\sum\limits_{n = 2}^\infty  {\left[ {\xi (u,v,\alpha ,\lambda )\left| {{a_n}{c_n}} \right|} \right]}  + \sum\limits_{n = 1}^\infty  {\left[ {\eta (u,v,\alpha ,\lambda )\left| {{b_n}{d_n}} \right|} \right]}
\end{align*}
\begin{align*}
\le \sum\limits_{n = 2}^\infty  {\xi (u,v,\alpha ,\lambda )\left| {{a_n}} \right|}  + \sum\limits_{n = 1}^\infty  {\eta (u,v,\alpha ,\lambda )\left| {{b_n}} \right|}
\end{align*}
\begin{align*}
=\sum\limits_{n = 2}^\infty  {\left[ {{{\left[ {1 + \left( {n - 1} \right)\lambda } \right]}^v} + \frac{{\left( {{{\left[ {1 + \left( {n - 1} \right)\lambda } \right]}^u} - {{\left[ {1 + \left( {n - 1} \right)\lambda } \right]}^v}} \right)\left( {1 + k} \right)}}{{1 - \alpha }}} \right]} \left| {{a_n}} \right| \\
+ \sum\limits_{n = 1}^\infty  {\left[ {{{\left( { - 1} \right)}^{v - u}}{{\left[ {1 + \left( {n - 1} \right)\lambda } \right]}^v} + \frac{{\left( {{{\left[ {1 + \left( {n - 1} \right)\lambda } \right]}^u} - {{\left( { - 1} \right)}^{v - u}}{{\left[ {1 + \left( {n - 1} \right)\lambda } \right]}^v}} \right)\left( {1 + k} \right)}}{{1 - \alpha }}} \right]} \left| {{b_n}} \right|\\
\end{align*}
$\le 1$.\\
This proves the result.
\end{proof}
\begin{theorem}(Convex Combination)\\
The family $k-UTH(u, v, \alpha, \lambda)$ is closed under convex combination.
\end{theorem}
\begin{proof}
For $j=1,2,3,..$, suppose that $f_u \in k-UTH(u, v, \alpha, \lambda)$, where
\begin{align*}
{f_{u,j}}(z) = z - \sum\limits_{n = 2}^\infty  {\left| {{a_{j,n}}} \right|} \,{z^n} + {\left( { - 1} \right)^{u - 1}}\sum\limits_{n = 1}^\infty  {\left| {{b_{j,n}}} \right|} \,{\overline z ^n}
\end{align*}
Then by Theorem 2.1
\begin{align}
\sum\limits_{n = 2}^\infty  {\xi (u,v,\alpha ,\lambda )\left| {{a_{j,n}}} \right|}  + \sum\limits_{n = 1}^\infty  {\eta (u,v,\alpha ,\lambda )\left| {{b_{j,n}}} \right|}  \le 1.
\end{align}
For $\sum\limits_{j = 1}^\infty  {{t_j} = 1} ,\,0 \le {t_j} \le 1$, the convex combination of ${f_{u,j}}\,\left( {j = 1,2,3,...} \right)$ may be written as
\begin{align*}
\sum\limits_{j = 1}^\infty  {{t_j}{f_{u,j}}(z) = z - \sum\limits_{n = 2}^\infty  {\left( {\sum\limits_{j = 1}^\infty  {{t_j}\left| {{a_{j,n}}} \right|} } \right)\,{z^n} + } } {\left( { - 1} \right)^{u - 1}}\sum\limits_{n = 1}^\infty  {\left( {\sum\limits_{j = 1}^\infty  {{t_j}\left| {{b_{j,n}}} \right|} } \right)} \,{\overline z ^n}
\end{align*}
Then, by (2.5),
\begin{align*}
\sum\limits_{n = 2}^\infty  {\xi (u,v,\alpha ,\lambda )\left( {\sum\limits_{j = 1}^\infty  {{t_j}\left| {{a_{j,n}}} \right|} } \right)}  + \sum\limits_{n = 1}^\infty  {\eta (u,v,\alpha ,\lambda )\left( {\sum\limits_{j = 1}^\infty  {{t_j}\left| {{b_{j,n}}} \right|} } \right)}
\end{align*}
\begin{align*}
&=\sum\limits_{j = 1}^\infty  {{t_j}} \left[ {\sum\limits_{n = 2}^\infty  {\xi (u,v,\alpha ,\lambda )} \left| {{a_{j,n}}} \right| + \sum\limits_{n = 1}^\infty  {\eta (u,v,\alpha ,\lambda )\left| {{b_{j,n}}} \right|} } \right]\\
&\le\sum\limits_{j = 1}^\infty  {{t_j}}=1.
\end{align*}
and therefore, $\sum\limits_{j = 1}^\infty  {{t_j}{f_{u,j}}(z)} \in k-UTH(u, v, \alpha, \lambda)$.
\end{proof}

\pagebreak

\end{document}